\documentclass{amsart}
\usepackage{amsfonts,  bbm}

\newtheorem{theorem}{Theorem}[section]
\theoremstyle{plain}
\newtheorem{corollary}[theorem]{Corollary}
\newtheorem{lemma}[theorem]{Lemma}
\newtheorem{proposition}[theorem]{Proposition}
\theoremstyle{remark}

\numberwithin{equation}{section}

\newcommand{\tr}{\operatorname{tr}}
\newcommand{\otr}{\operatorname{0-tr}}
\newcommand{\vol}{\operatorname{vol}}
\newcommand{\ovol}{\operatorname{0-vol}}

\newcommand{\inj}{\operatorname{inj}}
\newcommand{\diam}{\operatorname{diam}}
\newcommand{\re}{\operatorname{Re}}

\newcommand{\res}{\operatorname{Res}}
\newcommand{\rank}{\operatorname{rank}}

\newcommand{\bbR}{\mathbb{R}}
\newcommand{\bbH}{\mathbb{H}}
\newcommand{\bbC}{\mathbb{C}}
\newcommand{\bbZ}{\mathbb{Z}}
\newcommand{\bbN}{\mathbb{N}}
\newcommand{\calA}{\mathcal{A}}

\newcommand{\calR}{\mathcal{R}}
\newcommand{\calL}{\mathcal{L}}
\newcommand{\calM}{\mathcal{M}}
\newcommand{\calF}{\mathcal{F}}
\newcommand{\cinf}{C^\infty}
\newcommand{\del}{\partial}

\newcommand{\barX}{{\bar X}}
\newcommand{\bX}{{\partial_\infty X}}

\newcommand{\Oint}{\sideset{^{0\hskip-5pt}}{}\int}
\newcommand{\Ds}{(\Delta_g-s(n-s))}
\newcommand{\vep}{\varepsilon}
\newcommand{\tS}{\tilde{S}}

\newcommand{\hD}{\hat\Delta}
\newcommand{\hR}{\hat R}
\newcommand{\Drel}{D_{\rm rel}}
\newcommand{\Rsc}{\calR^{\rm sc}}

\newcommand{\nh}{\tfrac{n}{2}}
\newcommand{\norm}[1]{\Vert #1 \Vert}
\newcommand{\bnorm}[1]{\bigl\Vert #1 \bigr\Vert}

\begin{document}
\title[Inverse Scattering Results]{Inverse Scattering Results for Manifolds 
Hyperbolic near Infinity}
\author[Borthwick]{David Borthwick}
\address[Borthwick]{Department of Mathematics and Computer Science, Emory University,
Atlanta, Georgia 30322}
\email[Borthwick]{davidb@mathcs.emory.edu}
\author[Perry]{Peter A. Perry}
\address[Perry]{Department of Mathematics, University of Kentucky, Lexington, Kentucky 40506-0027}
\email[Perry]{perry@ms.uky.edu}
\thanks{Borthwick supported in part by NSF\ grant DMS-0901937.  
Perry supported in part by NSF grant DMS-0710477.}
\date{June 24, 2010}

\begin{abstract}
We study the inverse resonance problem for conformally compact manifolds which are
hyperbolic outside a compact set.  Our results include compactness of isoresonant metrics
in dimension two and of isophasal negatively curved metrics in dimension three.  In dimensions
four or higher we prove topological finiteness theorems under the negative curvature assumption.
\end{abstract}
\maketitle
\tableofcontents

\section{Introduction}\label{intro.sec}
The inverse problem of recovering an asymptotically hyperbolic metric from 
the associated scattering data has many possible variants, depending on how 
much knowledge is assumed.    It is well-known
that the resonance set does not determine an asymptotically hyperbolic manifold
completely, even in the exactly hyperbolic case.  See, for example,
Guillop\'e-Zworski \cite[Remark 2.15]{GZ:1997},
Brooks--Gornet--Perry \cite{BGP:2000}, Brooks-Davidovitch \cite{BD:2003}, and the survey paper 
Gordon-Perry-Schueth \cite{GPS:2005}.
One can however obtain strong positive results by assuming knowledge of the
scattering matrix itself.  For surfaces, a result of Lassas-Uhlmann \cite{LU:2001}
shows that the scattering matrix at the point $s=1$ determines the metric up to isometry.
The corresponding result for even dimensional conformally compact Einstein manifolds
was proven by Guillarmou-S\'a~Barreto \cite{GSB:2009}.
Another recent inverse result of S\'a~Barreto \cite{SB:2005} shows that an 
asymptotically hyperbolic manifold is completely determined by scattering matrix
at all energies.   Note that one must fix the boundary at infinity to make sense of the
assumption that two scattering matrices are equal.

In between these two extremes, another standard assumption in scattering theory
is that the metrics are \emph{isophasal}, meaning that they share the 
same \emph{scattering phase}.  Defining the scattering phase requires some 
regularization of scattering determinants.  In the even dimensional asymptotically
hyperbolic case, Guillarmou \cite{Gui:2005b} shows that a canonical regularization
can be defined.
In the odd dimensional case, we can only define a relative scattering phase between two
manifolds that are isometric near infinity.  
For all of the isoscattering examples cited above, the resonance sets can be identified 
because the respective scattering matrices are intertwined by transplantation operators.  
For hyperbolic surfaces the transplantation method gives examples that are
isophasal, as noted in \cite[Remark 2.15]{GZ:1997}.   The three-dimensional
isoscattering pairs are not necessarily isometric near infinity, so it's not even clear that 
the relative scattering phase is well-defined in these cases.

For conformally compact manifolds which are hyperbolic near infinity
(i.e. outside a compact set), the Hadamard factorization 
of the relative scattering determinant from Borthwick 
\cite[Prop.\ 7.2]{Borthwick:2008} shows that the resonance set determines the scattering
phase (relative to some fixed background metric) up to a polynomial of degree $n+1$.  
Thus, assuming a background metric is fixed, the isophasal condition is only slightly stronger than 
isoresonance, in the sense that it requires the equality of only a few additional parameters.  

The purpose of this note is to prove topological finiteness and geometric compactness results
in the context of conformally compact manifolds hyperbolic near infinity, for
isoresonant classes in even dimensions and isophasal classes in odd dimensions.

For $(X, g)$ conformally compact and hyperbolic near infinity, we let $\dim X = n+1$
and denote by $\Delta_g$ the positive Laplacian associated to $g$.   The 
resolvent $R_g(s) := \Ds^{-1}$ has a meromorphic continuation to $s\in\bbC$
with poles of finite rank \cite{MM:1987, GZ:1995b}.   
The \textit{resonance set} $\calR_g$ is the set of poles of $R_g(s)$, counted according
to the multiplicity given by   
$$
m_g(\zeta) := \rank \res_\zeta R_g(s).
$$

Resonances are closely related to the poles of the scattering matrix $S_g(s)$, 
defined as in \cite{JS:2000, GZ:2003}.  Let $\rho$ be a boundary defining function
for the conformal compactification $\barX$.
For $\re s = \nh$, $s\ne \nh$, a function $f_1\in \cinf(\bX)$
determines a unique solution of $\Ds u = 0$ such that
$$
u \sim \rho^{n-s} f_1 + \rho^s f_2
$$
as $\rho \to 0$, with $f_2\in \cinf(\bX)$
This defines the map $S_g(s): f_1 \mapsto f_2$, which extends meromorphically to $s\in \bbC$ as
a family of pseudodifferential operators of order $2s-n$.
To define a scattering determinant, we will fix a background metric $g_0$ and
use $S_{g_0}$ as a reference operator. 
If metrics $g, g_0$ agree to $O(\rho^\infty)$, then the product $S_g(s) S_{g_0}(s) - I$ is 
smoothing \cite{JS:2000} and so the relative scattering determinant,
\begin{equation}\label{tau.def}
\tau(s) := \det S_g(s) S_{g_0}(s)^{-1},
\end{equation}
is well-defined as a Fredholm determinant.  When restricted to the critical line $\re s = \nh$,
we have $|\tau(s)| = 1$, and the relative
scattering phase is a real-valued function (for real $\xi$) defined by
\begin{equation}\label{sigma.def}
\sigma(\xi) := \frac{i}{2\pi} \log \tau(\nh + i\xi),
\end{equation}
with branches chosen so that $\sigma(\xi)$ is continuous starting from $\sigma(0) = 0$.
By the symmetry properties of the scattering matrix, $\sigma(-\xi) = - \sigma(\xi)$. 
The scattering matrices depend on the choice of $\rho$, but $\tau(s)$ and 
$\sigma(\xi)$ are invariantly defined.

To state our results, fix a conformally compact manifold $(X_0, g_0)$ of dimension $n+1$ with 
a compact subset $K_0 \subset X_0$ such that $g_0$ is hyperbolic outside $K_0$ 
(meaning sectional curvatures $=-1$).   We wish to allow arbitrary metric perturbations 
within $K_0$, and so consider the class 
\begin{equation}\label{calM.def}
\calM(X_0, g_0, K_0) := \Bigl\{ (X,g) :\; (X-K, g) \cong (X_0 - K_0, g_0) \text{ for some }K\subset X\Bigr\},
\end{equation}
where $\cong$ denotes Riemannian isometry.
For each $X_0, g_0$ we will fix a boundary defining function $\rho$ and then use this 
same function for the entire class $\calM(X_0, g_0, K_0)$.

Naturally, the strongest results are possible in the case of surfaces:
\begin{theorem}\label{2d.cpt.thm}
Fix $X_0, g_0, K_0$ as above with $\dim X_0 = 2$.  If $\calA \subset \calM(X_0, g_0, K_0)$ 
is a collection of surfaces $(X,g)$ that share a common resonance set $\calR$, then $\calA$ 
is compact in the $\cinf$ topology.
\end{theorem}
This of course is analogous to the well-known result of Osgood-Phillips-Sarnak \cite{OPS:1988}
for compact surfaces.  And it is a considerable improvement over the comparable result of
Borthwick-Judge-Perry \cite[Thm~1.4]{BJP:2003}, for which the metric perturbations
were restricted to conformal deformations with compactly supported conformal parameter.
(See \S\ref{2d.sec} for some explanation of the improvement.)

In three dimensions we require more restrictive geometric assumptions and
more scattering data to produce a comparable result:
\begin{theorem}\label{3d.cpt.thm}
Fix $(X_0, g_0)$ and $K_0\subset X_0$ as above with $\dim X_0 = 3$.  Assume that 
$\calA \subset \calM(X_0, g_0, K_0)$ is a set of 3-manifolds $(X,g)$ with negative
sectional curvatures which share a common scattering phase.
Then $\calA$ is compact in the $\cinf$ topology.
\end{theorem}
Note that the isophasal condition could be expressed without reference to the 
scattering matrix of $(X_0, g_0)$ by requiring that the relative scattering phase 
between any pair of manifolds in $\calA$ is zero.  In practice it will be more 
convenient to define relative phases $\sigma_g(\xi)$ with respect to the fixed
background $g_0$.

Theorem \ref{3d.cpt.thm} is closely analogous to compactness results obtained
for isospectral compact 3-manifolds by Anderson \cite{Anderson:1991} and 
Brooks-Perry-Petersen \cite{BPP:1992}.
In dimensions greater than three, the conclusions are limited to topological finiteness,
just as in the corresponding results of \cite{BPP:1992}.
\begin{theorem}\label{finite.thm}
Fix $X_0, g_0, K_0$ as above with $\dim X_0 = n+1 \ge 4$.  Assume that 
$\calA \subset \calM(X_0, g_0, K_0)$ is a set of $(n+1)$-manifolds $(X,g)$ with negative
sectional curvatures which share either 
\begin{itemize}
\item a common resonance set $\calR$ if $\dim X$ is even, or
\item a common scattering phase $\sigma(\xi)$ if $\dim X$ is odd.
\end{itemize}
Then $\calA$ contains only finitely many homeomorphism types, and for $\dim X>4$
at most finitely many diffeomorphism types.
\end{theorem}

The paper is organized as follows.  In \S\ref{poisson.sec}--\ref{heat.sec} we review
the scattering theory and the various results that allow one to deduce geometric information 
from it.  The proof of Theorem~\ref{finite.thm} is given in \S\ref{finite.sec}.
In \S\ref{cpt.sec} we review some geometric compactness results and apply these to give 
the proofs of Theorems~\ref{2d.cpt.thm} and \ref{3d.cpt.thm}.  The proof for surfaces
is the most complicated, in that we must establish curvature bounds without any
control of the injectivity radius at the outset.  This part of the proof, which is based on
conformal uniformization, is deferred to \S\ref{2d.sec}.

\vskip12pt\noindent
\textbf{Acknowledgment.}   The authors are grateful for support from the
Mathematical Sciences Research Institute, where a portion of this work was carried out.
We also thank Pierre Albin for various helpful comments and corrections.

\section{Poisson formula}\label{poisson.sec}

Resonances are closely related to the poles of the scattering matrix $S_g(s)$,
defined as in \cite{JS:2000, GZ:2003}.  This operator has infinite-rank poles, so
to define multiplicities of scattering poles, we use a renormalized scattering 
matrix of order zero given by 
\begin{equation}\label{tS.def}
\tS_g(s) := \frac{\Gamma(s-\nh)}{\Gamma(\nh - s)}
\Lambda^{n/2-s} S_g(s) \Lambda^{n/2-s}.
\end{equation}
where
$$
\Lambda := \frac12 (\Delta_h+1)^{1/2}.
$$
This renormalization makes $\tS_g(s)$ into a meromorphic family of Fredholm operators with 
poles of finite rank.  The multiplicity at a pole or zero of $S_g(s)$ is then defined by
$$
\nu_g(\zeta) := - \tr \bigl[\res_\zeta \tS_g'(s) \tS_g(s)^{-1}\bigr]
$$
(with poles counted positively to match the resonances).
The dependence of $\tS_g(s)$ on the boundary defining function $\rho$ is wiped out by
the trace, so that $\nu_g(\zeta)$ is invariantly defined.

The scattering multiplicities are related to the resonance multiplicities by results of
Guillop\'e-Zworksi \cite{GZ:1997}, Borthwick-Perry \cite{BP:2002} and 
Guillarmou \cite{Gui:2005} (with a restriction that was later removed in \cite{GN:2006}):
\begin{equation}\label{nu.mumu}
\nu_g(\zeta) = m_g(\zeta) - m_g(n-\zeta) + \sum_{k\in \bbN} \Bigl( \mathbbm{1}_{n/2 - k}(\zeta) 
-\mathbbm{1}_{n/2+k}(\zeta) \Bigr) d_k,
\end{equation}
where $\mathbbm{1}_p$ denotes the characteristic function on $\{p\}$ and
$$
d_k := \dim \ker \tS_{g}(\tfrac{n}2+k).
$$
From Graham-Zworski \cite{GZ:2003} it follows that the $d_k$'s are invariants of the 
conformal structure induced on $\bX$ by the metric $\rho^2 g$. 
For surfaces ($n=1$), the $d_k$ terms always vanish \cite[Lemma 8.6]{Borthwick}.  
But in higher dimensions they may occur and even saturate the resonance counting
function (see \cite{GN:2006} or \cite{Borthwick:2008}).  

To state certain results, such as the Poisson formula, we need to incorporate these
extra scattering poles into a \emph{scattering resonance} set,
$$
\Rsc_g := \calR_g \cup \bigcup_{k=1}^\infty 
\Bigr\{\tfrac{n}2 - k \text{ with multiplicity } d_k\Bigr\}.
$$
For any inverse scattering problem, it makes sense to assume that the $d_k$'s
are fixed, since they depend only on the structure at infinity. 

We will state inverse scattering results in two different contexts.  
First, we showing that certain geometric information that can be deduced solely from $\Rsc_g$, 
without assuming knowledge of $(X_0,g_0)$.  The catch is that for this purpose 
we must assume that $(X_0, g_0)$ is exactly hyperbolic.  Later in the section, we'll give inverse
results that apply within $\calM(X_0, g_0, K_0)$.  This is the context of \S\ref{intro.sec}, 
for which we assume knowledge of $(X_0,g_0)$ and $\calR_{g_0}$, 
but drop the assumption that the background is exactly hyperbolic.

In the case of a compactly supported perturbation of a conformally compact
hyperbolic metric, Borthwick \cite{Borthwick:2008} gave a Poisson formula for resonances
that relates the regularized wave trace, defined as a distribution on $\bbR$ by
$$
\Theta_g(t) := \otr \left[ \cos \left(t \sqrt{\smash[b]{\Delta_g - n^2/4}}\,\right) \right],
$$
to a sum over $\Rsc_g$.  The assumption the the background is exactly hyperbolic
allows contributions from the background metric to be cancelled from both sides of
a relative Poisson formula, yielding a result that has no explicit dependence on 
$(X_0, g_0)$ or $\Rsc_{g_0}$.

\begin{theorem}[Poisson formula]\label{poisson.thm}
Let $(X,g)$ be a compactly supported perturbation of a conformally
compact hyperbolic manifold.  Then, in a distributional sense on $\bbR - \{0\}$, 
$$
\Theta_g(t)
= \frac12 \sum_{\zeta\in \Rsc_g} e^{(\zeta-n/2)|t|} - A(X) \frac{\cosh t/2}{(2\sinh |t|/2)^{n+1}},
$$
where
$$
A(X) := \begin{cases}0 & n \text{ odd }(\dim X\>\text{is even}),  \\
\chi(X) & n \text{ even }(\dim X\>\text{is odd}).  \end{cases}
$$
\end{theorem}
\noindent
Note that in odd dimensions we could also write $A(X)$ as $\tfrac12 \chi(\bX)$.

In two dimensions this formula is due to Guillop\'{e} and Zworski \cite{GZ:1999}, 
and the requirement for an exactly hyperbolic background metric is not necessary
for that case.  For hyperbolic manifolds of any dimension it was proved by Guillarmou and Naud
\cite{GN:2006}.   The result as stated here is Borthwick \cite[Thm.~1.2]{Borthwick:2008}

\begin{corollary}\label{wtrace.cor}
Assume $(X,g)$ is a compactly supported perturbation of a conformally
compact hyperbolic manifold.
In the even-dimensional case ($n$ odd), the set $\Rsc_g$ determines the wave 0-trace 
as a distribution on $\bbR$, and fixes $\ovol(X,g)$ in particular.
In odd dimensions ($n$ even), $\Rsc_g$ determines $\chi(X)$ and the
restriction of the wave trace to $t\ne 0$.
\end{corollary}
\begin{proof}
Joshi and S\'{a} Barreto \cite{JS:2001} showed that the asymptotic expansion
of the wave $0$-trace at $t=0$ has the same form as found by Duistermaat-Guillemin
\cite{DG:1975}.  That is,
if $\psi \in \cinf_0(\bbR)$ has support 
in a sufficiently small neighborhood of $0$ and $\psi = 1$ in some smaller
neighborhood of $0$, then 
\begin{equation}\label{wave.asymp}
\int_{-\infty}^\infty e^{-it\xi} \psi(t) \Theta_g(t)\>dt \sim \sum_{k=0}^\infty 
a_k |\xi|^{n-2k},
\end{equation}
where 
$$
a_0 =   \frac{2^{-n} \pi^{-\frac{n-1}2}}{\Gamma(\frac{n+1}{2})} \ovol(X,g).
$$ 

In even dimensions, the powers $|\xi|^{n-2k}$ correspond to singularities of the
form $t^{-n-1+2k}$ (homogeneous regularization).  Thus the singularities 
are detectable in the behavior of the wave $0$-trace as $t\to 0_+$.  By the Poisson formula,
$\Rsc_g$ determines the wave trace completely for $t\ne 0$, and so the wave coefficients 
$\{a_k\}$ are also fixed by $\Rsc_g$.

In the odd dimensional case ($n$ even),
$|\xi|^{n-2k}$ corresponds to $\delta^{(n-2k)}(t)$ when $n-2k \ge 0$.
Thus, the singularity of the wave 0-trace at $t=0$ is not computable from $\Rsc_g$.  
(Indeed, in odd dimensions $a_0$ depends on the choice of boundary defining function $\rho$,
so to obtain $a_0$ from $\Rsc_g$ is impossible a priori.)
Since the wave-trace singularities are localized at $t=0$,  one sees only the blowup caused
by the $\chi(X)$ term as $t \to 0_+$.  Hence $\chi(X)$ is fixed by $\Rsc_g$.
\end{proof}

Joshi and S\'{a} Barreto \cite{JS:2001} also showed that the wave $0$-trace for an asymptotically
hyperbolic manifold has singularities for $t\neq0$ contained in the set of
lengths of closed geodesics of $X$.  In the case when the sectional curvatures of $(X,g)$ 
are strictly negative, Rowlett \cite[Thm~1.1]{Rowlett:2008} has recently
refined this result to show that, for $t \ge \vep >0 $ we have
\begin{equation}\label{wave.sing}
\Theta_g(t) = \sum_{\ell \in \calL_g} \sum_{k=1}^\infty \frac{\ell}{\sqrt{|\det 1-P_\ell^k|}} \delta(t - k\ell) + R(t),
\end{equation}
where $\calL_g$ is the primitive length spectrum of $(X,g)$, $P_\ell^k$ is the $k$-times
around Poincar\'e map for the geodesic associated to $\ell$, and the remainder $R(t)$
is smooth and bounded on $[\vep,\infty)$.  This immediately leads the following:
\begin{corollary}\label{inj.cor}
Assuming that $(X, g)$ is a compactly supported perturbation of a conformally
compact hyperbolic manifold with strictly negative sectional curvatures,
the resonance set $\Rsc_g$ determines the length spectrum of $(X,g)$, and 
in particular fixes the injectivity radius $\inj(X,g)$.
\end{corollary}

\noindent
(Note that under the negative curvature assumption, $\inj(X,g)$ is equal to
half the length of the shortest closed geodesic.)

\bigbreak
We now turn to the results needed for the applications given in \S\ref{intro.sec},
for which we can assume full knowledge of the fixed background $(X_0,g_0)$.  
In this situation, we can start from a relative Poisson formula,
which does not require the background to be exactly hyperbolic.  
\begin{theorem}
For $(X, g) \subset \calM(X_0, g_0, K_0)$ defined as in (\ref{calM.def}), 
where $(X_0, g_0)$ is conformally compact and hyperbolic near infinity, we have
$$
\Theta_g(t) - \Theta_{g_0}(t) = \frac12 \sum_{\zeta\in \calR_g} e^{(\zeta-n/2)|t|} 
- \frac12 \sum_{\zeta\in \calR_{g_0}} e^{(\zeta-n/2)|t|}.
$$
\end{theorem}
\begin{proof}
We define the meromorphic function,
$$
\Upsilon_g(s) := (2s-n) \otr [R_g(s) - R_g(n-s)],
$$
for $s \notin \bbZ/2$.  By \cite[Lemma~7.1 and Prop.~7.2]{Borthwick:2008},
we have
\begin{equation}\label{rel.ups}
\Upsilon_g(s) - \Upsilon_{g_0}(s) = \del_s \log \left[e^{q(s)} 
\frac{P_g(n-s)}{P_g(s)} \frac{P_{g_0}(s)}{P_{g_0}(n-s)}\right],
\end{equation}
where $P_*(s)$ denotes the Hadamard product over the resonance set $\calR_*$,
and $q(s)$ is a polynomial.  (We can use $\calR_*$ rather than $\Rsc_*$ here,
because the extra $d_k$ terms are canceled by the background.)
On the other hand, \cite[Eq.~(8.2) and Lemma~8.1]{Borthwick:2008}
show that $\Upsilon_g(s)$ is essentially the inverse Fourier transform of the continuous part
of the wave trace.  Taking the Fourier transform of (\ref{rel.ups}), exactly as in the
proof of \cite[Thm.~1.2]{Borthwick:2008}, yields the formula given above.
\end{proof}

\begin{corollary}\label{rel.wtrace.cor}
Assuming $\dim X_0$ is even, for metrics in $\calM(X_0, g_0, K_0)$ the resonance set 
$\calR_g$ determines $vol(K,g)$.   In any dimension, for metrics of strictly negative sectional
curvatures in $\calM(X_0, g_0, K_0)$, the resonance set $\calR_g$ determines 
$\inj(X,g)$.
\end{corollary}
\begin{proof}
Using (\ref{wave.asymp}), in the even dimensional case we can deduce
$$
\ovol(X,g) - \ovol(X_0, g_0) = \vol(K,g) - \vol(K_0, g_0)
$$
from $\calR_g$ and $\calR_{g_0}$.  Hence, within $\calM(X_0, g_0, K_0)$
we see that $\calR_g$ determines $\vol(K,g)$.
Similarly, since $\Theta_{g_0}(t)$ is fixed within $\calM(X_0, g_0, K_0)$,
from (\ref{wave.sing}) we see that $\calR_g$ determines the length spectrum
for metrics of negative sectional curvature, and hence the injectivity radius.
\end{proof}

%%%%%%%%%%%
\section{Relative scattering phase}\label{scphase.sec}

For $X \in \calM(X_0,K_0,g_0)$,
the relative scattering determinant $\tau(s)$ and scattering phase $\sigma(\xi)$
were defined in (\ref{tau.def}) and (\ref{sigma.def}), respectively.  
Since $\tau(s)$ is meromorphic, fixing $\sigma(\xi)$ determines $\tau(s)$ as well.
Define the Weierstrass product
\begin{equation}\label{Pg.def}
P_g(s) := \prod_{\zeta\in \Rsc_g} E\Bigl(\frac{s}{\zeta}, n+1\Bigr),
\end{equation}
where $E(w,k)$ is an elementary factor, 
$$
E(w,k) := (1-w) e^{w+w^2/2+\dots+w^k/k}.
$$
Let $P_{g_0}(s)$ be the corresponding product for $\Rsc_{g_0}$.
By \cite[Prop.~7.2]{Borthwick:2008},
\begin{equation}\label{tau.factor}
\tau(s) = e^{q(s)} \frac{P_g(n-s)}{P_g(s)}  \frac{P_{g_0}(s)}{P_{g_0}(n-s)},
\end{equation}
where $q(s)$ is a polynomial of degree at most $n+1$.  
The coefficients of $q(s)$, which has the symmetry $q(s) = - q(n-s)$, are the extra parameters
that we fix by assuming equality of scattering phases instead of resonance sets.
In the other direction, the factorization formula (\ref{tau.factor}) makes it clear that 
$\sigma(\xi)$ determines $\Rsc_g$, modulo the fixed background $\Rsc_{g_0}$.

Another important formula for the relative scattering phase connects it to the
(regularized) traces of the spectral resolutions.  For $s \ne \bbZ/2$ we have
$$
\frac{\del \sigma}{\del \xi}(\xi) = \frac{i\xi}{\pi} \Bigl( \otr \bigl[R_g(\nh+i\xi) - R_g(\nh-i\xi)\bigr] 
- \otr \bigl[R_{g_0}(\nh+i\xi) - R_{g_0}(\nh-i\xi) \bigr]\Bigr).
$$
By the functional calculus, the two terms on the right are the Fourier transforms of
the continuous parts of the respective regularized wave traces, except at $\xi = 0$,
where the $0$-trace can have an anomaly.
By \cite[(8.1--2)]{Borthwick:2008}, we deduce the following:
\begin{proposition}\label{scphase.wavetr}
For $(X_0, g_0)$ conformally compact and hyperbolic near infinity
and $(X,g) \in \calM(X_0, g_0, K_0)$, the 
relative scattering phase $\sigma(\xi)$ determines the relative wave trace
$\Theta_g(t) - \Theta_{g_0}(t)$, as a distribution for $t\in \bbR$.
\end{proposition}
Note that the big singularity of the wave trace at $t=0$ is included in this result, because
it corresponds to the behavior of $\sigma(\xi)$ as $|\xi| \to \infty$.

%%%%%%%%%
\section{Relative heat invariants}\label{heat.sec}

Suppose that $(X, g)$ is conformally compact and hyperbolic near infinity,
and let $H_g(t;z,z')$ denote the heat kernel  associated to $\Delta_g$.
The restriction of the heat kernel to the diagonal has the usual local expansion
as $t\to 0$,
\begin{equation}\label{loc.heat}
H_g(t;z,z) \sim  t^{-\frac{n+1}2} \sum_{j=0}^\infty  t^j \alpha_j(g; z).
\end{equation}
In our setting, the heat operator is not trace class, and 
the local geometric invariants $\alpha_j(g)$ are not integrable over $(X,g)$.
To obtain global invariants we subtract off contributions from the background metric
$(X_0,g_0)$.  Since $\alpha_j(g)$ agrees with $\alpha_j(g_0)$ on $X-K \cong X_0 - K_0$,
we define the relative heat invariant as
\begin{equation}\label{agjj.def}
a_j(g,g_0) := \int_K \alpha_j(g)\>dg - \int_{K_0} \alpha_j(g_0)\>dg_0.
\end{equation}

By the formula connecting the heat and wave operators,
\begin{equation}\label{heat.wave}
e^{-u (\Delta_g - n^2/4)} = \frac{1}{\sqrt{\pi u}} \int_0^\infty e^{-t^2/4u}  
\cos \left(t \sqrt{\smash[b]{\Delta_g - n^2/4}}\,\right)\>dt,
\end{equation}
and the characterization of the wave kernel in Joshi-S\'a Barreto \cite{JS:2001},
we can see that the heat kernel has a well-defined $0$-trace (i.e.~its kernel is
polyhomogeneous in $\rho$ as $\rho \to 0$).

We could try to define regularized heat invariants directly from the $0$-trace
of the $\alpha_j(g)$'s.  For conformally compact Einstein manifolds, Albin \cite{Albin:2009} 
shows that that these $0$-traces give well-defined invariants.  In our situation it
is simpler to consider only the expansion of the relative heat trace, and the corresponding
relative heat invariants, for which
any possible dependence on the regularization scheme is effectively canceled.

\begin{proposition}\label{heattr.prop}
The difference of heat $0$-traces admits an expansion in terms of relative heat invariants,
$$
\otr \bigl( e^{-t\Delta_g}\bigr) - \otr \bigl( e^{-t\Delta_{g_0}} \bigr) 
\sim t^{-\frac{n+1}2} \sum_{j=0}^\infty  t^j a_j(g, g_0).
$$
\end{proposition}

\begin{proof}
By the local form of the heat expansion (\ref{loc.heat}), we see immediately that
$$
\int_K H_g(t;z,z)\>dg(z) - \int_K H_{g_0}(t;z,z)\>dg_0(z) \sim 
t^{-\frac{n+1}2} \sum_{j=0}^\infty  t^j a_j(g, g_0).
$$
Hence the goal is to show that
\begin{equation}\label{relh.claim}
\Oint_{X_0-K_0} \bigl[ H_{g}(t;z,z) - H_{g_0}(t;z,z) \bigr] dg_0(z) = O(t^\infty),
\end{equation}
as $t \to 0$, where we implicitly make use of the isometry $(X-K, g) \cong (X_0 - K-0, g_0)$
to combine the two $0$-integrals. 

To estimate (\ref{relh.claim}) near infinity we introduce
cutoff functions $\psi_1, \psi_2 \in \cinf(X_0-K_0)$, both zero on $\del K_0$ and 1 near infinity,
with $\psi_1 = 1$ on some open neighborhood of the support of $\psi_2$.  
After pullback by isometry (which we suppress from the notation), 
we can regard $\psi_2 e^{t\Delta_g} \psi_1$ as an operator on $L^2(X_0-K_0, dg_0)$.
By integrating
$$
\frac{d}{du} \Bigl[\psi_2 e^{-u\Delta_g} \psi_1 e^{-(t-u)\Delta_{g_0}} \psi_1 \Bigr] 
= - \psi_2 e^{-u\Delta_g} [\Delta_{g_0}, \psi_1] e^{-(t-u)\Delta_{g_0}} \psi_1,
$$
we obtain a cutoff version of Duhamel's formula, 
$$
\psi_2 e^{-t\Delta_g} \psi_1 - \psi_1 e^{-t\Delta_{g_0}} \psi_2
= - \int_0^t \psi_2 e^{-u\Delta_g} [\Delta_{g_0}, \psi_1] e^{-(t-u) \Delta_{g_0}} \psi_2\>du.
$$
Choose $\eta \in \cinf_0(X_0-K_0)$ such that $\eta = 1$ on the support of
$[\Delta_{g_0},\psi_1]$ and so that the supports of $\eta$ and $\psi_2$ are separated by distance
$\delta>0$.  We can rewrite the above formula as
\begin{equation}\label{psiHpsi}
\psi_2 e^{-t\Delta_g} \psi_1 - \psi_1 e^{-t\Delta_{g_0}} \psi_2
= -  \int_0^t A_1(u) A_2(t-u)\>du,
\end{equation}
where 
$$
A_1(u) := \psi_2 e^{-u\Delta_g} \eta,
$$
and
$$
A_2(u) :=  [\Delta_{g_0}, \psi_1] e^{-u \Delta_{g_0}} \psi_2.
$$

Using the estimates of Cheng-Li-Yau \cite[Cor.~8]{CLY:1981} for the heat kernel on
complete manifolds with bounded curvatures, 
we can estimate the kernels of the $A_i(u)$ by
$$
A_i(u;z,w) \le C_i u^{-(n+i)/2} e^{-cd(z,w)^2/u}.
$$
Since the kernels are smooth and decay rapidly at infinity, we conclude that 
the $A_i(u)$'s are Hilbert-Schmidt.  Moreover, because the $d(z,w)\ge \delta$ in the
supports of the cutoffs, we can estimate the Hilbert-Schmidt norms by
$$
\norm{A_i(u)}_{2} \le C_i e^{-c\delta^2/u}.
$$
From (\ref{psiHpsi}) we can then estimate the trace norm
$$
\bnorm{\psi_2 e^{-t\Delta_g} \psi_1 - \psi_1 e^{-t\Delta_{g_0}} \psi_2}_1
= O(t^{\infty}).
$$
This shows that the $0$-integral in (\ref{relh.claim}) is a convergent integral and that
$$
\int_{X_0-K_0} \psi_2(z) \bigl[ H_{g}(t;z,z) - H_{g_0}(t;z,z) \bigr] dg_0(z) = O(t^\infty),
$$

Finally, on $X_0-K_0$, we have $\alpha_j(g) = \alpha_j(g_0)$, so that the estimate,
$$
\int_{X_0-K_0} (1-\psi_2(z)) \bigl[ H_{g}(t;z,z) - H_{g_0}(t;z,z) \bigr] dg_0(z) = O(t^\infty),
$$
follows from the local heat expansion (\ref{loc.heat}).
\end{proof}

If we assume knowledge of the the relative scattering phase, then it is relatively easy
to recover relative heat invariants via the wave trace. 
\begin{proposition}\label{phase.heat}
For $(X_0, g_0)$ conformally compact and hyperbolic near infinity
and $(X,g) \in \calM(X_0, g_0, K_0)$, the relative scattering phase $\sigma(\xi)$ determines the 
relative heat invariants $a_j(g, g_0)$.
\end{proposition}
\begin{proof}
By Proposition~\ref{scphase.wavetr}, the relative scattering
phase determines the difference of the wave $0$-traces for $g$ and $g_0$.
Using the relation (\ref{heat.wave}) between the heat and wave operators, we can
then apply Proposition~\ref{heattr.prop} to recover the relative heat invariants.
\end{proof}

\bigbreak
In even dimensions we are able to get more information out of the resonance set, 
following the methods of \cite{BJP:2003}, with some restrictions on the background metric.
We will only make application of these results in dimension two
(see \S\ref{2d.sec}), but we may as well give the proof for any even dimension. 

For this argument, assume that $(X,h)$ is conformally compact hyperbolic
and that $g$ is another metric on $X$ that agrees with $h$ to order $\rho^2$.
(This easing of the restriction that $g$ and $h$ agree outside a compact set will
actually be required for the arguments based on conformal uniformization in \S\ref{2d.sec}.)
Let $L^2(X)$ denote the space of square-integrable half-densities, 
with $\hD_g$ and $\hD_{h}$ the Laplacians on $L^2(X)$ associated to the respective 
metrics.  We deduce that $e^{-t\hD_{g}} - e^{-t\hD_{h}}$ is a trace class operator on $L^2(X)$
from Duhamel's formula,
$$
e^{-t\hD_{g}} - e^{-t\hD_{h}} = \int_0^t e^{-u\hD_{g}} (\hD_{g} - \hD_{h}) e^{-(t-u)\hD_{g}}\>du.
$$
In this context the relative heat trace expansion is given by
\begin{equation}\label{gh.heat}
\tr \Bigl[ e^{-t\hD_{g}} - e^{-t\hD_{h}} \Bigr] \sim  t^{-\frac{n+1}2} \sum_{j=0}^\infty  t^j b_j,
\end{equation}
where
\begin{equation}\label{gh.coeff}
b_j := \lim_{\vep \to 0} \left[ \int_{\{\rho \ge \vep\}} \alpha_j(g)\>dg -  \int_{\{\rho \ge \vep\}} \alpha_j(h)\>dh  \right].
\end{equation}

The parametrix construction from \cite{GZ:1995b} shows that the
operator $\hR_g(s)^m - \hR_{h}(s)^m$ is trace class on $L^2(X)$ for $\re s >n$
with $m=(n+3)/2$.  For $\re w\ge m$ and $\re s >n$ define the relative
zeta function
$$
\zeta(w,s) := \tr \bigl[ \hR_g(s)^w - \hR_{h}(s)^w \bigr].
$$
In terms of heat operators, we have
\begin{equation}\label{zeta.heat}
\zeta(w,s) = \frac{1}{\Gamma(w)} \int_0^\infty t^w e^{ts(n-s)} \tr \bigl[ e^{-t\hD_{g}} 
- e^{-t\hD_{h}} \bigr] \frac{dt}{t}.
\end{equation}
The heat expansions as $t \to 0$ can be used to show that $\zeta(w,s)$ extends meromorphically
to $\re w > -1$, with simple poles at $w=\tfrac{n+1}2,\tfrac{n-1}2, \dots$, ending at 1 for $n$ odd 
and continuing to negative half-integers for $n$ even.   In any dimension $\zeta(w,s)$ is analytic at $w=0$,
and so the relative determinant,
$$
\Drel(s) := \exp \bigl[-\del_w\zeta(w,s)|_{w=0}\bigr],
$$
is well-defined for $\re s > n$.

Let $Z_h(s)$ denote the Selberg zeta function for $(X,h)$.
Patterson-Perry \cite[Thm.~1.9]{PP:2001} proved the factorization formula
\begin{equation}\label{zeta.factor}
Z_{h}(s) = e^{p_1(s)} G_\infty(s)^{-\chi(X)} P_{h}(s),
\end{equation}
where $p_1(s)$ is a polynomial of degree at most $n+1$ and 
$$
G_\infty(s) = s \prod_{k=1}^\infty E(-\tfrac{s}{k}, n+1)^{h_n(k)},
$$
with
$$
h_n(k) := (2k+n) \frac{(k+1)\dots (k+n-1)}{n!}.
$$
The formula (\ref{zeta.factor}) remains valid even when $(X,h) = \bbH^{n+1}$; in this case $Z_h(s) := 1$,
and the poles of $G_\infty(s)^{-1}$ cancel the zeroes of $P_h(s)$.

From the proof of \cite[Prop~7.2]{Borthwick:2008} we see that
$$
\Drel(s) := e^{p_2(s)} \frac{P_g(s)}{P_{h}(s)},
$$
with $p_2(s)$ also a polynomial of degree at most $n+1$.
Thus we have
\begin{equation}\label{Drel.factor}
\Drel(s) :=  \frac{e^{p(s)} P_g(s)}{Z_{h}(s) G_\infty(s)^{\chi(X)}},
\end{equation}
for $p(s)$ a polynomial of degree at most $n+1$.

\begin{proposition}\label{res.det}
Suppose that $(X, h)$ is a conformally compact hyperbolic metric with $\dim X$ even, 
and $g$ is a metric hyperbolic near infinity that agrees with $h$ to order $\rho^2$.
Then the Euler characteristic $\chi(X)$ and 
the resonance set $\Rsc_g$ together determine the product $\Drel(s)Z_h(s)$ and
all of the relative heat invariants $b_j$ defined by (\ref{gh.coeff}).
When $\dim X = 2$, the set $\calR_g = \Rsc_g$ alone determines $\chi(X)$, $\Drel(s)Z_h(s)$,
and the relative heat invariants.
\end{proposition}
\begin{proof}
We examine the asymptotic expansion of $\log \Drel(s)$ as $\re s \to \infty$.  By
(\ref{zeta.heat}) and the heat expansion, we have
\begin{equation}\label{Drel.asymp}
\begin{split}
\log \Drel(s) & \sim \sum_{j=0}^{\frac{n+1}2} c_{n,j} b_j [s(s-n)]^{\frac{n+1}2 - j} \log [s(s-n)] \\
&\qquad + \sum_{j>\frac{n+1}2} c_{n,j} b_j [s(s-n)]^{\frac{n+1}2 - j},
\end{split}
\end{equation}
where the $c_{n,j}$'s are nonzero combinatorial constants.

On the other hand, consider the factorization (\ref{Drel.factor}).  The log of $Z_{h}(s)$
decays exponentially as $\re s \to \infty$.  Thus $\chi(X_0)$ and $\Rsc_g$ together
determine the asymptotic expansion of $p(s) + \log \Drel(s)$ as $\re s \to \infty$, where
$p(s)$ is the polynomial appearing in (\ref{Drel.factor}).  Because of the log terms
in (\ref{Drel.asymp}), both the heat invariants and the
coefficients of $p(s)$ are fixed by this expansion.

The $n=1$ case of this result was proven in \cite[Prop.~5.8]{BJP:2003}.
in this case, the known asymptotics of $\log G_\infty(s)$ and the
vanishing of the first relative heat invariant (by Gauss-Bonnet), allow
the Euler characteristic also to be determined from $\Rsc_g$.
\end{proof}

The amusing feature of Proposition~\ref{res.det} is that no information
on $\Rsc_h$ is needed for the result, because of the structure of the Selberg
zeta function.  
In odd dimensions, the corresponding argument breaks down because the asymptotic 
formula corresponding to (\ref{Drel.asymp}) is 
$$
\log \Drel(s) \sim \sum_{j=0}^\infty c_{n,j} b_j [s(s-n)]^{\frac{n+1}2 - j},
$$
i.e.~there are no logarithmic terms.
The absence of such terms means we cannot rule out cancelation between the
coefficients of $p(s)$ and the relative heat invariants $b_0, \dots, b_{n/2}$.

\section{Finiteness of topological types}\label{finite.sec}

For compact manifolds dimensions greater than 3, the heat invariants do not contain 
enough information to  establish $C^k$ bounds on the curvatures.   This problem
of course persists in the non-compact case.  However, we can certainly use
spectral information to control the topological type, following arguments of 
\cite{BPP:1992}.  The crucial result is the following:
\begin{theorem}[Grove-Petersen-Wu \cite{GPW:1990}, Thm.~C]\label{gpw.thm}
The class of closed Riemannian $m$-manifolds $M$ with injectivity radius bounded below
and volume bounded above contains at most finitely many homeomorphism
types if $m \ge 4$, and only finitely many diffeomorphism types if $m\ge 5$.
\end{theorem}

Fix an asymptotically hyperbolic manifold $(X_0, g_0)$ with boundary defining function
$\rho$ and a compact subset $K_0 \subset X_0$.   Let $\calM(X_0, g_0, K_0)$ denote the 
class of manifolds $X,g$ such that $(X-K, g) \cong (X_0 - K_0, g_0)$ for some compact $K\subset X$.
We will assume that $0$-volumes for elements of  $\calM(X_0, g_0, K_0)$ are defined by
boundary defining functions that agree with $\rho$ on $X-K$.  

\begin{corollary}\label{gpw.cor}
The set of manifolds in $\calM(X_0, g_0, K_0)$ with injectivity radius bounded below and 
$\vol(K,g)$ bounded above contains at most finitely many homeomorphism
types if $\dim X_0 \ge 4$, and only finitely many diffeomorphism types if $\dim X_0\ge 5$.
\end{corollary}
\begin{proof}
Suppose that we glue two copies of $K_0$ together along a neck $N_0$, diffeomorphic to
$\bX \times [-1,1]$, to form a compact manifold $D_0$, 
with metric $\tilde g_0$ defined as a smooth extension of the $g_0$ 
metric on each copy of $K_0$.  For some $\delta>0$ we may assume that
a region near the edges of $(N_0, \tilde g_0)$, defined by
$$
Z_{2\delta} := \Bigl\{ p \in N_0:\> d(p, \del N_0) \le 2\delta \Bigr\} \subset N_0,
$$ 
is isomorphic to the corresponding region of $(X_0,g_0)$.  

We can use the same neck $(N_0, \tilde g_0)$ to form the corresponding
double $(D, \tilde g)$ for any $(X, g) \in \calM(X_0, g_0, K_0)$.  The volume of this
double is controlled by
\begin{equation}\label{vol.D}
\vol (D, \tilde g) \le 2 \vol(K,g) + \vol(N_0, \tilde g_0),
\end{equation}
which is bounded above by assumption.

As for the injectivity radius, we claim that
\begin{equation}\label{inj.D}
\inj(D, \tilde g) \ge c,
\end{equation}
where $c$ depends only on $\inj(X,g)$, the fixed geometry of $(N_0,\tilde g_0)$, 
and $\delta$.  Consider first a point $p \in D- N_0$.  If a geodesic loop 
originating at $p$ lies entirely within $K \cup Z_\delta$ (using either copy of $K$), 
then its length is bounded
below by $2\inj(X,g)$.  On the other hand, if a point of the geodesic loop 
intersects $N_0 - Z_\delta$, then the length of the loop is greater than $2\delta$.
The same reasoning applies to any segment connecting $p$ to a conjugate point,
so we conclude that $\inj(p)$ satisfies the bound (\ref{inj.D}) in this case.
The argument starting from $p \in N_0 - Z_\delta$ is virtually identical.

This leaves the case of $p \in Z_\delta$.  If geodesic loop originating at $p$
has length shorter than $\delta$, then it lies completely within $K \cup Z_{2\delta}$
and this length is bounded below by $2\inj(X,g)$.  Since $(X_0-K_0, g_0)$ has
negative curvature, there are no conjugate points within $Z_{2\delta}$.
Thus if a segment joining $p$ to a conjugate point is shorter than $\delta$,
it must lie completely within $K \cup Z_{2\delta}$.  The length of this segment
is then bounded below by $\inj(X,g)$.
This completes the proof of (\ref{inj.D}).

Using (\ref{vol.D}) and (\ref{inj.D}), the result now follows from Theorem~\ref{gpw.thm}.
\end{proof}

\bigbreak
It is now straightforward to combine these results with the spectral results 
from the preceding sections.  Note that fixing $(X_0, g_0)$ fixes the $d_k$
contributions to $\Rsc_g$, so it does not matter in the statement of Theorem~\ref{finite.thm} 
whether we specify $\calR_g$ or $\Rsc_g$ for the even dimensional case.
\begin{proof}[Proof of Theorem~\ref{finite.thm}]
In even dimensions, fixing $\calR$ controls $\vol(K,g)$ and the injectivity radius
by Corollary \ref{rel.wtrace.cor}.  The result then follows immediately
from Corollary~\ref{gpw.cor}.

In odd dimensions, extra information is required because the resonance set does not
fix the $0$-volume.  (This would be impossible, because the $0$-volume can
be made arbitrarily large through the choice of $\rho$.)  To control the volume we must
fix the scattering phase and appeal to Proposition~\ref{phase.heat}.
Since the zeroth relative heat invariant is $\vol(K,g) - \vol(K_0, g_0)$,
this fixes $\vol(K,g)$ for metrics in $\calM(X_0, g_0, K_0)$
Because the scattering phase determines $\Rsc_g$
(relative to the fixed background set $\Rsc_{g_0}$), Corollary~\ref{rel.wtrace.cor} gives
control over the injectivity radius.  The result thus follows by Corollary~\ref{gpw.cor}.
\end{proof}

\section{Geometric compactness theorems}\label{cpt.sec}

To prove $\cinf$ compactness of a particular class of metrics, we seek to apply 
the following $\cinf$ version of the Cheeger compactness theorem:
\begin{theorem}[Kasue \cite{Kasue:1989}, Croke \cite{Croke:1980}]\label{kasue.thm}
Let $(M_j, g_j)$ be a sequence of compact Riemannian manifolds with uniform bounds
of the form:
$$
\vol(M_j, g_j) \le C, \qquad \inj(M_j, g_j)  \ge c, \qquad \sup |\nabla^k \text{\rm Ricc}(g_j)| \le C_k.
$$
Then, after passing to a subsequence, there exists a manifold $M_\infty$ with 
diffeomorphisms $\varphi_{j}: M_\infty \to M_{j}$ such that the metrics $\varphi_{j}^*g_{j}$ 
converge in the $\cinf$ topology on $M_\infty$.
\end{theorem}

This is a modification of the compactness theorem of Kasue \cite{Kasue:1989}, which 
assumes a uniform bound on the diameters of $(M_j,g_j)$.   
(The original version is more refined, yielding $C^{k,\alpha}$ compactness 
based on control of derivatives of the curvature
up to order $k$.)   
Since the spectral data give control of the volumes of the cores $(K,g)$, 
it is more convenient for us to switch from diameter to volume.    
This link is provided by Croke \cite[Cor.~15]{Croke:1980},
who proves that for any compact $m$-dimensional Riemannian manifold $(M, g)$,
$$
\diam(M,g) \le \frac{2m^m \Omega_m}{\Omega_{m-1}} \frac{\vol(M,g)}{\inj(M,g)^{m-1}},
$$
with $\Omega_m$ the volume of $S^m$.

It is tempting to try to generalize Theorem~\ref{kasue.thm} to the case of
even-dimensional asymptotically hyperbolic manifolds, by replacing the volume estimate 
with a bound on the $0$-volume.  (There's no hope of this in odd dimensions because
the $0$-volume is not invariantly defined.)  But at least for surfaces we can see immediately
that this does not work.  Consider a pair of pants with boundary geodesics of length 
$\ell_1, \ell_2, \ell_3$
and funnels attached to each of these.  As $\ell_1 \to \infty$ the sequence clearly diverges,
but curvature is constant, injectivity radius remains equal to $\min(\ell_2,\ell_3)$, 
and the $0$-volume is also constant at $2\pi$.  
The obvious doubling argument that one might try to
extend Theorem~\ref{kasue.thm} fails here because the injectivity radius 
of the doubled surface may approach zero.

\subsection{Isoresonant compactness in dimension two}

The two dimensional application of Theorem~\ref{kasue.thm} is based
on the following intermediate result:
\begin{proposition}\label{Kbounds}
Suppose $(X,g)$ is a conformally compact surface hyperbolic near infinity, with 
$K(g)$ denoting the Gaussian curvature.
We have bounds
$$
\inj(X,g) \ge c,\qquad \sup|\nabla_g^k K(g)| \le C_k,
$$
for any $k = 0,1,2,\dots$, where the constants $c>0$ and $C_k>0$ depend only
on the resonant set $\calR_g$.
\end{proposition}

We will defer the somewhat technical proof of Proposition~\ref{Kbounds}
to \S\ref{2d.sec}.

\bigbreak
\begin{proof}[Proof of Theorem \ref{2d.cpt.thm}]
Let $\calA$ denote a collection of surfaces as described in the statement of the theorem. 
By Proposition~\ref{res.det}, $\ovol(X,g)$ is constant over $\calA$.  Hence
$\vol(K, g)$ is constant as well.  If we form the doubles $(D, \tilde g)$, by gluing
two copies of each compact regions $(K,g)$ along a common neck $N$, then 
we produce a corresponding class $\tilde\calA$ of compact surfaces 
$(D, \tilde g)$.  These metrics share a fixed volume and the $C^k$ curvature bounds
from Proposition~\ref{Kbounds} extend directly because the same neck is used for every
case.  As in the proof of Corollary~\ref{gpw.cor}, the injectivity radius is bounded
below in terms of the lower bound on $\inj(X,g)$ from Proposition~\ref{Kbounds},
the width of the neck, and the curvature in the neck.  

Starting from a sequence $\{(X, g_k)\} \subset \calA$, we form doubles $(D, \tilde g_k)$.
By Theorem~\ref{kasue.thm} we can assume, after passing to a subsequence,  that there exist
diffeomorphisms $\varphi_k: D \to D$ such that $\{\varphi_k^*\tilde g_k\}$ converges in the $\cinf$
topology on $D$ to some metric $\tilde g_\infty$.  In order to apply this result to the original sequence, 
we need to make sure that the diffeomorphisms $\varphi_k$ converge to the identity on the neck.

Let $d_N$ denote the distance function corresponding to the metric on the neck,
which is fixed independently of $k$.
Suppose $p_1,..., p_n$ are points in $N$, chosen so that the distance functions 
$d_N(p_i, \cdot)$ collectively provide good sets of coordinates covering all of $N$.  Since $D$ is
compact, by passing to a subsequence of $\{ \varphi_k \}$ we can assume
that $\varphi_k^{-1}(p_i)$ converges to some point $p_{i,\infty} \in D$
as $k \to \infty$, for each $i=1, ..., n$.  For $q\in N$ we have
$$
d_{\varphi_k^*\tilde g_k}(\varphi_k^{-1}(p_i), \varphi_k^{-1}(q)) = d_N(p_i, q).
$$
Because the metrics $\varphi_k^*\tilde g_k \to \tilde g_\infty$ and 
$\varphi_k^{-1}(p_i) \to p_{i,\infty}$, this implies that 
\begin{equation}\label{dtilde.conv}
\lim_{k\to\infty} d_{\tilde g_\infty}(p_{i,\infty}, \varphi_k^{-1}(q)) = d_N(p_i, q).
\end{equation}
Since all of the neck metrics are isometric, the functions $d_{\tilde g_\infty}(p_{i,\infty}, \cdot)$
also provide good sets of coordinates, we conclude from (\ref{dtilde.conv}) that 
$\varphi_k^{-1}(q)$ converges to some point $q_\infty$ such that
$$
d_{\tilde g_\infty}(p_{i,\infty}, q_\infty) = d_N(p_i, q).
$$
This argument shows that the restriction of $\varphi_k^{-1}$ to $N$ converges to a map 
$\psi:N \to N$ which is just 
identity map between the respective coordinate systems defined 
by $\{d_N(p_i, \cdot)\}$ and $\{d_N(p_{i,\infty}, \cdot)\}$.

We can extend $\psi$ to a diffeomorphism $D \to D$ in some arbitrary way and,
after replacing $\varphi_k$ by $\psi\circ\varphi_k$, we can assume that
$\varphi_k$ converges to the identity on $N$.  Then we obtain a solution to the
original problem by restricting the resulting sequence to $K$.
\end{proof}

\bigbreak
\subsection{Isophasal compactness in dimension three}

Our compactness argument is actually somewhat easier for $\dim X = 3$, because 
the extra hypothesis of negative curvature gives us control over the injectivity radius
immediately from Corollary~\ref{rel.wtrace.cor}.  Since $\vol(K,g)$ is fixed by the first relative
heat invariant, the doubling trick is essentially all that we need to adapt standard arguments
from the compact case.

The one point to clear up is that we can produce bounds on the Sobolev constants
of the compact doubles $(D, \tilde g)$, using spectral information from the 
original spaces $(X,g)$.
The results of Brooks-Perry-Petersen \cite[\S2]{BPP:1992} do not apply verbatim,
because they assume knowledge of the eigenvalue spectrum of $(D, \tilde g)$.
Adapting these arguments to our case is a relatively simple matter; we include the
details for the sake of clarity of exposition.

\begin{theorem}\label{sob.thm}
Let $(M, g)$ be a compact $m$-dimensional Riemannian manifold, and assume
$$
\vol(M, g) \le C, \qquad \inj(M, g)  \ge c.
$$
Then for each $p$ the constant $C_p$ in the Sobolev inequalities: for $f\in\cinf(M)$
$$
\norm{f}_{\frac{pm}{m-p}} \le C_p \bigl(\norm{f}_p + \norm{\nabla f}_p \bigr)\qquad 1 \le p < m,
$$
and
$$
\norm{f}_\infty  \le C_p \bigl(\norm{f}_p + \norm{\nabla f}_p \bigr)\qquad p>m,
$$
is bounded above by a constant that depends only on $p$, $c$, and $C$.
\end{theorem}
\begin{proof}
By \cite[Thm.~14]{Croke:1980}, for any $r \le \tfrac12 \inj(M,g)$ we have
\begin{equation}\label{isoper}
\frac{\vol(\del B(p; r))^m}{\vol(B(r))^{m-1}} \ge \frac{2^{m-1} \Omega_{m-1}^m}{\Omega_{m}^{m-1}},
\end{equation}
where $\Omega_m$ is the volume of $S^m$.
Moreover, this bound can be integrated \cite[Prop.~15]{Croke:1980}, yielding,
for any $r \le \tfrac12 \inj(M,g)$,
\begin{equation}\label{volbound}
\vol(B(p;r)) \ge  \frac{2^{m-1} \Omega_{m-1}^m}{m^m \Omega_{m}^{m-1}}\> r^m.
\end{equation}
Fix $r = \tfrac12 \inj(M,g)$. 
If we pack $M$ with a maximal collection of disjoint balls $B(p_j, r/2)$,
$j=1, \dots, k$, then (\ref{volbound}) gives a bound on the number $k$ of such balls:
\begin{equation}\label{kbound}
k \le \frac{2m^m \Omega_{m}^{m-1}}{\Omega_{m-1}^m} \frac{\vol(M,g)}{r^m}.
\end{equation}

For $f \in \cinf_0(B(p;r))$, we can now apply \cite[Cor.~2.1]{BPP:1992}, which gives the 
claimed Sobolev bounds in this case with constants controlled by virtue of
(\ref{isoper}) and (\ref{volbound}).   A simple partition of unity argument (see
\cite[pp.~78--9]{BPP:1992}) applied to the cover $\{B(p_j, r)\}_{j=1}^k$, together
with the bound (\ref{kbound}), then extends the result to $f\in \cinf(M)$.
\end{proof}

\bigbreak
\begin{proof}[Proof of Theorem~\ref{3d.cpt.thm}]
Let $\calA \subset \calM(X_0, K_0, g_0)$ be a collection as in the statement of the theorem.
According to Proposition~\ref{phase.heat}, fixing the relative scattering
phase fixes the relative heat invariants.  Since the background metric
is held constant, this in turn fixes the integrals
\begin{equation}\label{heatK}
a_{j,K}(g) := \int_K \alpha_j(g)\>dg,
\end{equation}
where $\alpha_j(g)$ is the $j$-th local heat invariant of $\Delta_g$, as in 
(\ref{loc.heat}).  In particular, the $j=0$ case shows that
$\vol(K,g)$ is fixed for $(X,g) \in \calA$.
By the assumption of negative curvature, the injectivity radius of $(X,g)$ is fixed 
by Corollary~\ref{rel.wtrace.cor}.

Now we form the collection $\tilde\calA$ of doubles $(D,\tilde g)$ as in the proof
of Corollary~\ref{gpw.cor}.  The volume and injectivity radius of 
any $(D,\tilde g)\in \tilde\calA$ are controlled just as in that proof.   
Theorem~\ref{sob.thm} therefore gives uniform 
control of the Sobolev constants of $(D, \tilde g)$.  And using the constants $a_{j,K}(g)$,
together with the corresponding integrals over the fixed neck, we see that the heat invariants
of $(D, \tilde g)$ are fixed for the collection $\tilde\calA$.  

The final step is to apply the bootstrap argument to produce $C^k$ bounds on the
Ricci tensor from the heat invariants, using the Sobolev inequalities.  
For compact manifolds of dimension three 
this was done in Brooks--Petersen--Perry \cite[\S5]{BPP:1992}, and we will not
repeat the details here.  (See also 
the nice expository account of this argument
in Brooks \cite{Brooks:1997}.)
\end{proof}

\bigbreak
\section{Curvature estimates in dimension two}\label{2d.sec}

The main issue in two dimensions is to control the injectivity radius without
assuming the curvature is negative.  The tool for accomplishing this is conformal 
uniformization, which was also the basis for the results of Osgood-Phillips-Sarnak \cite{OPS:1988}
as well as Borthwick-Judge-Perry \cite{BJP:2003}.  

For conformally compact manifolds, the relevant uniformization theorem follows
from the work of Mazzeo-Taylor \cite{MT:2002}.  There results show in particular that any metric
$\bar g$ on $\bar X$ is conformally related to a unique complete hyperbolic metric,
with control of the boundary regularity of the conformal factor.
By \cite[Cor.~4.2]{BJP:2003}, we can assume an extra order of vanishing of the conformal
factor when $K(g) = -1 + O(\rho^2)$.  In particular we have the following corollary to 
the Mazzeo-Taylor result:
\begin{proposition}\label{MT.thm}
If $(X,g)$ is a conformally compact surface hyperbolic near infinity, then there exists
a unique $\varphi \in \rho^2 \cinf(\bar X)$ such that 
$$
g = e^{2\varphi}h,
$$ 
where $h$ is a complete hyperbolic metric on $X$.
\end{proposition}

The compactness arguments cited above \cite{BJP:2003, OPS:1988} rely
on the production of a convergent subsequence of uniformizing hyperbolic metrics, which 
allows reduction to the case of a single fixed background metric $h$. In the non-compact 
case  \cite{BJP:2003} this approach requires unfortunate extra restrictions: compact support
for the $\varphi$ and upper bounds on the diameters of funnels for the $h$.

The argument presented in this section differs from the previous approaches
(including Osgood-Phillips-Sarnak) in that the
background metric $h$ is never fixed.  Instead, we rely on uniform control of the resolvent
$R_h(s)$ to turn $H^k(X, dh)$ bounds on $\varphi$ into $C^k$ bounds on $K(g)$.
We can then exploit the fact that $K(g)+1$ is compactly supported and avoid any
restriction on the support of $\varphi$.

It is quite possible that the approach presented here could be extended to surfaces with cusps.
The conformal uniformization results one would need to use have recently been proved
by Ji-Mazzeo-Sesum \cite{JMS:2007} (for finite volume only) and
Albin-Aldana-Rochon \cite{AAR:2009} (for the general case).

Suppose we take $(X,g), h, \varphi$ as in Proposition~\ref{MT.thm} and
apply Proposition~\ref{res.det} to the pair $g,h$.  This shows that
$\chi(X)$ and the relative heat invariants $b_j$, defined in (\ref{gh.coeff}),
are determined by $\calR_g$.  The zeroth relative heat invariant is
\begin{equation}\label{b0.gh}
b_0 =  \frac{1}{4\pi} \int_X (e^{2\varphi} - 1)\>dh.
\end{equation}
Proposition~\ref{res.det} also tells us that the product $\Drel(s)Z_h(s)$ is an invariant 
of $\calR_g$.  In particular, the invariant quantity
$$
d_0 := \log \Drel(1)Z_h(1),
$$
will play an important role here.
This is because of the Polyakov formula \cite{Polyakov:1981, Alvarez:1983}, 
which was extended to the asymptotically hyperbolic context in 
\cite[Prop.~1.2]{BJP:2003}:
\begin{equation}\label{polyakov}
\log \Drel(1) = - \frac{1}{6\pi} \int_X \bigl(\tfrac12 |\nabla_h \varphi|^2 - \varphi \bigr)\>dh.
\end{equation}
We should note that, in contrast to the compact case  \cite{OPS:1988},
$\log \Drel(1)$ is not an invariant of $\calR_g$.  Fortunately, the quantity
$d_0$ makes a suitable replacement.

For this section it will be convenient to use the notation 
$$
A \preceq B \quad\Longleftrightarrow \quad A \le CB,
$$ 
where $C>0$ depends only on the invariants
of $\calR_g$, namely $d_0$ and $b_0, b_1, \dots$.  For example,
we claim that
$$
\log \Drel(1) \succeq 1.
$$
To prove this, we note that the product formula for the Selberg zeta function,
\begin{equation}\label{Zdef}
Z_h(1) := \prod_{\ell \in \calL_h} \prod_{k=1}^{\infty}\left[  1- e^{-k\ell(\gamma)}\right],
\end{equation}
where $\calL_h$ denotes the primitive length spectrum of $(X,h)$,
converges in some neighborhood of $1$. 
(The hyperbolic surface $(X,h)$ has infinite area, so
the exponent of convergence for the associated Fuchsian group is strictly less than $1$.)  
For $(X,h) \cong \bbH^2$ we set $Z_h(s) := 1$.  In all other cases,
the convergence of (\ref{Zdef}) implies that $Z_h(1) \in (0,1)$. 
Hence we have a lower bound
for $\log \Drel(1)$ that depends only on $d_0$.

\begin{lemma}\label{phi.bounds}
For $g,h$ as given by Proposition~\ref{MT.thm}, we have bounds 
\begin{equation}\label{3bounds}
\left|\int_X \varphi\>dh \right| \preceq 1, \qquad \int_X |\nabla_h \varphi|^2\>dh \preceq 1,
\qquad \int_X |\varphi|^2\>dh \preceq 1,
\end{equation}
along with
\begin{equation}\label{Xh.bounds}
\inj(X,h) \succeq 1,\qquad \inf \sigma(\Delta_h) \succeq 1.
\end{equation}
The constants in these bounds depend only on the invariants $b_0$ and $d_0$.  
\end{lemma}
\begin{proof}
The first two bounds were obtained in the proof of \cite[Thm.~1.4]{BJP:2003}, but we
recall the details for the convenience of the reader.  
For $\vep >0$ set $X_\vep := \{\rho \ge 0\} \subset X$ and
$$
V_\vep := \vol(\{X_\vep, h)
$$
Applying Jensen's inequality with the convex function $F(x) = e^{2x}-1$ and the
probability measure $V_\vep^{-1} \>dh$ on $X_\vep$ gives
\[
\begin{split}
 \int_{X_\vep} \varphi\>dh 
& \le \frac{V_\vep}2 \log \left[1+ V_\vep^{-1} \int_{X_\vep} (e^{2\varphi}-1)\>dh \right] \\
& \le \frac12  \int_{X_\vep} (e^{2\varphi}-1)\>dh,
\end{split}
\]
where in the second line we just use $\log (1+x) \le x$.
Taking $\vep \to 0$ and comparing to (\ref{b0.gh}) gives
\begin{equation}\label{jenson.b0}
\int_X \varphi\>dh \le 2\pi b_0.
\end{equation}
From (\ref{polyakov}) and the fact that $\log \Drel(1) \ge d_0$ we then deduce
\begin{equation}\label{poly.phi}
\begin{split}
\frac{1}{6\pi} \int_X \varphi \>dh & = \frac{1}{12\pi} \int_X |\nabla_h \varphi|^2\>dh +
\log \Drel(1) \\
& \ge d_0.
\end{split}
\end{equation}
Together, (\ref{jenson.b0}) and (\ref{poly.phi}) give the first bound in (\ref{3bounds}).

We can now use the first bound to eliminate the $\varphi$ term from the 
Polyakov formula (\ref{polyakov}).  This yields the second bound, in the form
$$
\int_X |\nabla_h \varphi|^2 \>dh \le  4\pi b_0 - 12\pi d_0,
$$
as well as the useful estimate
\begin{equation}\label{logz.bd}
-\log Z_h(1) \le \frac{b_0}{3} - d_0.
\end{equation}
If $\ell_0(h) := \inf \calL_h$ then by (\ref{Zdef}) we have
$$
Z_h(1) \le 1 - e^{-\ell_0(h)},
$$
and so (\ref{logz.bd}) gives a lower bound 
\begin{equation}\label{inj.bound}
\inj(X,h) = \frac{\ell_0(h)}2  \succeq 1.  
\end{equation}

For the remainder of the argument, we apply a result of Dodziuk et al.~\cite[Thm.~1.1$'$]{DPRS:1987}, 
which allows one to estimate small eigenvalues of an infinite-area 
hyperbolic surface in terms of lengths of chains of disjoint simple closed geodesics.  In its simplest form,
this result implies
\begin{equation}\label{dprs}
\inf \sigma(\Delta_h) \succeq \ell_0(h),
\end{equation}
where the constant depends only on the topology of $X$.  This gives the second half of (\ref{Xh.bounds}).
The third bound in (\ref{3bounds}) now follows from second bound and the Poincar\'e inequality,
\begin{equation}\label{poinc.ineq}
\int_X |\varphi|^2\>dh  \le \frac{1}{\inf \sigma(\Delta_h)} \int_X |\nabla_h \varphi|^2\>dh.
\end{equation}
\end{proof}

One very useful consequence of the lower bound on the bottom of the spectrum of $\Delta_h$
is that it gives uniform control of the heat-kernel $H_h(t,z,w)$ of $\Delta_h$.  The results of
Davies-Mandouvalos \cite[Thm.~5.4]{DM:1988} yield the following estimate:
\begin{equation}\label{Ht.bound}
H_h(t,z,w) \le C_0 t^{-1} e^{-at} e^{-d(x,w)^2/Dt},
\end{equation}
for any $0< a \le \inf \sigma(\Delta_h)$ and $D >4$.  The constant 
$C_0$ depends only on the choice of $a$ and $D$.
Lemma \ref{phi.bounds} thus shows that (\ref{Ht.bound}) holds with constants that
depend only on $b_0$ and $d_0$.

At this point we've gotten all the information we can out of $b_0$.  And $b_1=0$ because 
$a_1(g) = a_1(h) = -2\pi\chi(X)$.  So the next step is to bring in the second relative heat invariant,
\begin{equation}\label{b2.gh}
b_2 = \frac{1}{60\pi} \int_X \Bigr[ e^{-2\varphi} (\Delta_h\varphi-1)^2 -1 \Bigr]\>dh.
\end{equation}

\begin{lemma}\label{C0.bound}
For $\varphi$ as in Proposition~\ref{MT.thm},
$$
\sup_X |\varphi| \preceq 1,
$$
where the constant depends only on the invariants $b_0$, $b_2$, and $d_0$.
\end{lemma}
\begin{proof}
To handle $b_2$, we need a Trudinger-type inequality with suitable control of the constants.
By a theorem of Grigor$'$yan \cite{Grigoryan:1994}, the Davies--Mandouvalos bound (\ref{Ht.bound})
implies the Faber-Krahn inequality:
$$
\lambda_1(\Omega) \succeq \vol(\Omega)^{-1},
$$
for any precompact region $\Omega \subset X$.  This allows us to apply some very general
results on Sobolev inequalities due to Bakry et al.~\cite{BCLS:1995}.  
In particular, by \cite[Thm.~10.1]{BCLS:1995} the Faber-Krahn inequality is equivalent to
a family of bounds:
\begin{equation}\label{Ssr}
\norm{u}_r^r \le (C\norm{\nabla_h u}_2)^{r-s}\> \norm{u}_s^s,
\end{equation}
for any $0<s<r<\infty$, where $\norm{\cdot}_p$ refers to $L^p(X, dh)$.   The constant $C$ depends 
only on the Faber-Krahn constant, which in turn depends only on $b_0$ and $d_0$.
Setting $s=2$ and summing over the cases $r=2,3\dots$ leads immediately to a Trudinger inequality 
\cite[Thm.~3.4]{BCLS:1995}, 
\begin{equation}\label{trud}
\int_X \exp_2 (u)\>dh \le \frac{\norm{u}_2^2}{(C\norm{\nabla_h u}_2)^2} \>
\exp_2 (C\norm{\nabla_h u}_2),
\end{equation}
where $\exp_2(x) := e^x - 1 - x$.

With the Trudinger inequality we can use $b_2$ to control the $L^2$ norm of 
$e^{-\varphi}\Delta_h\varphi$.  
The expansion the of the formula (\ref{b2.gh}) for $b_2$ gives
\begin{equation}\label{ephidelta}
\bnorm{e^{-\varphi}\Delta_h\varphi}_2^2 \le 60\pi b_2 
+  \left| \int_X \bigl(e^{-2\varphi} -1\bigr)\>dh\right| 
+ 2 \left| \int_X e^{-2\varphi} \Delta_h\varphi\>dh \right|.
\end{equation}
Here the the second term on the right-hand side may be controlled using (\ref{trud})
and Lemma~\ref{phi.bounds},
\[
\begin{split}
\left| \int_X \bigl(e^{-2\varphi} -1\bigr)\>dh\right| & \le  \left| \int_X (-2\varphi)\>dh \right| 
+  \int_X \exp_2 (-2\varphi)\>dh \\
& \preceq \left| \int_X \varphi \>dh \right| + \norm{\varphi}_2^2\\
& \preceq 1.
\end{split}
\]
The third term of (\ref{ephidelta}) is handled similarly, starting from
\[
\begin{split}
\left| \int_X e^{-2\varphi} \Delta_h\varphi\>dh \right|
& =  \left| \int_X (e^{-2\varphi}-1) \Delta_h\varphi\>dh \right| \\
& \le \bnorm{e^{\varphi} - e^{-\varphi}}_2\> \bnorm{e^{-\varphi}\Delta_h\varphi}_2.
\end{split}
\]
Since
\[
\begin{split}
\bnorm{e^{\varphi} - e^{-\varphi}}_2^2 & = \int_X \Bigl[e^{2\varphi} - 2 + e^{-2\varphi}\Bigr]\>dh  \\
& = \int_X \Bigl[\exp_2(2\varphi) + \exp_2(-2\varphi)\Bigr]\>dh,
\end{split}
\]
this term can also be bounded by means of (\ref{trud}) and Lemma~\ref{phi.bounds}.   
Thus from (\ref{ephidelta}) we obtain
$$
\bnorm{e^{-\varphi}\Delta_h\varphi}_2^2 \preceq 1 +  \bnorm{e^{-\varphi}\Delta_h\varphi}_2,
$$
and we immediately deduce that
\begin{equation}\label{dphi}
\bnorm{e^{-\varphi}\Delta_h\varphi}_2 \preceq 1.
\end{equation}

The next step is to produce an $L^p$ estimate on $R_h(s;z,\cdot)$. 
For $\re s > \tfrac12$
we can estimate $R(s;z,w)$ using the heat kernel estimate (\ref{Ht.bound}) in the formula
\begin{equation}\label{res.heat}
R_h(s;z,w) = \int_0^\infty e^{s(1-s) t} H_h(t;z,w)\>dt.
\end{equation}
For convenience we set $s=2$ (although any $s>1$ would suffice for our argument).  
For $r := d_h(z,w) \ge 3$, we make the following estimate of (\ref{res.heat}) in terms of the constants 
$C_0, a, D$ appearing in (\ref{Ht.bound}):
\[
\begin{split}
R_h(2;z,w) & \le C_0 \int_0^{r/3}  t^{-1} e^{-(2+a)t} e^{-r^2/Dt}\>dt 
+ C_0 \int_{r/3}^\infty  t^{-1} e^{-(2+a)t} e^{-r^2/Dt}\>dt \\
& \le C_0 \int_{3 r}^\infty  e^{-u/D}\>du + C_0 \int_{r/3}^\infty e^{-(2+a)t} \>dt \\
& \le C_0e^{-3 r/D} + C_0 e^{-(2+a)r/3},
\end{split}
\]
where we substituted $u = r^2/t$ in the second line.   Assuming, as we may, that $D \le 9/2$,
this yields a uniform bound for $r \ge 3$,
$$
R_h(2;z,w) \le 2C_0 e^{-2r/3}. 
$$
For $r \le 3$, we can split up the integral (\ref{res.heat})
for $R_h(2;z,w)$ to obtain 
\[
\begin{split}
R_h(2;z,w) & \le C_0 \int_0^{r^2}  t^{-1} e^{-r^2/Dt}\>dt 
+ C_0 \int_{r^2}^9  t^{-1} \>dt +  C_0 \int_{9}^\infty e^{-at}\>dt\\
& \le C_1 - C_2 \log r,
\end{split}
\]
where $C_1$ and $C_2$ depend only on $C_0$ and $D$.
The point of keeping track of the constants in these calculations is to obtain 
estimates solely in terms of $r = d_h(z,w)$ and constants that depend on $b_0$ and $d_0$ but are 
otherwise independent of the uniformizing hyperbolic metric $h$. 

To control the $L_p$ norms uniformly in $z$, we lift $R_h(z,w)$ to $\bbH$ and let 
$\calF$ be a fundamental domain corresponding to $(X,h)$. Then to eliminate the 
$z$-dependence we enlarge the domain from $\calF$ to $\bbH^2$ 
and switch to geodesic polar coordinates centered at $z$:
\[
\begin{split}
\norm{R_h(2;z,\cdot)}_p^p & = \int_\calF \bigl|R_h(2;z,w)\bigr|^p\>dh \\
& \le \int_{\bbH^2} \bigl|R_h(2;z,w)\bigr|^p\>dh \\
& \le 2\pi \int_0^3 \Bigl[C_1 - C_2 \log r \Bigr]^p\>\sinh r\>dr \\
&\qquad + 2\pi \int_3^\infty (2C_0)^p e^{-2pr/3}\>\sinh r\>dr
\end{split}
\] 
The integrals are convergent for $p\ge 2$, so this establishes uniform estimates 
\begin{equation}\label{res.lp}
\norm{R_h(2;z,\cdot)}_p \preceq 1,\qquad\text{for }p\ge 2,
\end{equation}
where for each $p$ the constant depends only on $b_0$ and $d_0$.

We can now combine the estimates (\ref{dphi}) and (\ref{res.lp}) to control $\varphi$
pointwise, starting from
$$
\varphi(z) = \int_X R_h(2;z,w) (\Delta_h+2) \varphi(w)\>dh.
$$
This leads immediately to
\begin{equation}\label{supphi}
|\varphi(z)| \le \bnorm{R_h(2;z,\cdot)e^{\varphi}}_2 \> \bnorm{e^{-\varphi}(\Delta_h+2)\varphi}_2.
\end{equation}
To bound the first term in (\ref{supphi}), we use
\[
\begin{split}
\bnorm{R_h(2;z,\cdot)e^{\varphi}}_2 & \le \bnorm{R_h(2;z,\cdot)}_2 + \bnorm{R_h(2;z,\cdot)(e^{\varphi}-1)}_2 \\
& \le  \bnorm{R_h(2;z,\cdot)}_2 + \bnorm{R_h(2;z,\cdot)}_4\> \bnorm{e^{\varphi}-1}_4.
\end{split}
\]
By (\ref{res.lp}) and (\ref{trud}), the norms on the right are all bounded by constants that depend only
on $b_0$ and $d_0$.  For the second term in (\ref{supphi}), we have
\[
\begin{split}
\bnorm{e^{-\varphi}(\Delta_h+2)\varphi}_2 & \le \bnorm{e^{-\varphi}\Delta_h\varphi}_2 + 2\bnorm{e^{-\varphi}\varphi}_2 \\
& \le \bnorm{e^{-\varphi}\Delta_h\varphi}_2 + 2\bnorm{\varphi}_2+ 2\bnorm{(e^{-\varphi}-1)\varphi}_2 \\
& \le  \bnorm{e^{-\varphi}\Delta_h\varphi}_2 + 2\bnorm{\varphi}_2+ 2\bnorm{e^{-\varphi}-1}_4\> \norm{\varphi}_4.
\end{split}
\]
The first term is bounded by (\ref{dphi}), and $\norm{\varphi}_p$ is covered for $p \ge 2$
by Lemma~\ref{phi.bounds} together with (\ref{Ssr}).  It is also easy to bound $\norm{e^{-\varphi}-1}_4$
by means of (\ref{trud}) and Lemma~\ref{phi.bounds}, since
$$
(e^{-\varphi}-1)^4 = \exp_2(-4 \varphi) - 4 \exp_2(-3\varphi) + 6 \exp_2(-2\varphi) - 4 \exp_2(-\varphi).
$$
Hence, the terms on the right side of (\ref{supphi}) are bounded by constants
that depend only on $b_0$, $b_2$, and $d_0$, and the result is proved.
\end{proof}

\bigbreak
With control of the conformal factor $e^{2\varphi}$, we are able to control the 
lengths of geodesics in $(X,g)$:
\begin{corollary}\label{ell0.bound} 
Suppose $(X,g)$ is a conformally compact surface hyperbolic near infinity, and 
let $\ell_0(g)$ denote the length of the shortest closed geodesic.
Then we have
$$
\ell_0(g) \succeq 1,
$$
with a constant that depends only on  $b_0$, $b_2$, and $d_0$.
\end{corollary}
\begin{proof}
Suppose $\eta$ is a closed geodesic on $(X,g)$.  By Lemma~\ref{C0.bound},
we can estimate the $g$-length by
$$
\ell(\eta;g) \succeq \ell(\eta;h).
$$
Although $\eta$ will not be a $h$-geodesic in general, we still have the bound
$\ell(\eta;h) \ge  \ell_0(h)$.
Since $\ell_0(h)$ is bounded below in terms of $d_0$, this gives a lower bound
on $\ell(\eta; g)$ that depends only on $b_0, b_2$ and $d_0$.
\end{proof}

\bigbreak
\begin{proof}[Proof of Proposition~\ref{Kbounds}]
Since $K(g)$ is not integrable on $(X,g)$, for the sake of estimates it is convenient
to replace it by the compactly supported function
$$
\Psi := K(g) +1 = e^{-2\varphi} \Delta_h \varphi.
$$
To control $\norm{K(g)}_\infty$,
we seek to estimate $\norm{\Delta_h \Psi}_2$ and then remove the Laplacian using 
$R_h(2)$ as in the proof of Lemma~\ref{phi.bounds}. 

The third local heat invariant has the form
$$
\alpha_3(g) = c_1 |\nabla_g K(g)|^2 + c_2 K(g)^3,
$$
where $c_1 \ne 0$ according to \cite[Appendix]{OPS:1988}.
Thus the third relative invariant is
\begin{equation}\label{b3}
b_3 = c_1 \int_X |\nabla_g K(g)|^2 \>dg + c_2\int_X (K(g)^3e^{2\varphi} +1) \>dh
\end{equation}
By $g = e^{2\varphi} h$ we have
$$
\int_X |\nabla_g K(g)|^2 \>dg =  \int_X |\nabla_h K(g)|^2 \>dh = \norm{\nabla_h \Psi}_2^2.
$$
Noting that
$$
\int_X \Psi e^{2\varphi}\>dh = \int_X \Delta_h\varphi \>dh = 0,
$$
the second term in $b_3$ can be reduced to
$$
\int_X (K(g)^3e^{2\varphi} +1) \>dh = \int_X (\Psi^3 - 3\Psi^2)e^{2\varphi}\>dh - 4\pi b_0.
$$
Lemma~\ref{C0.bound} gives us control of $\sup |e^{2\varphi}|$, and
the combination of Lemma~\ref{C0.bound} and (\ref{dphi}) gives a bound on
$\norm{\Psi}_2$.  Thus from (\ref{b3}) we obtain
\begin{equation}\label{abpsi3}
\norm{\nabla_h \Psi}_2^2 \preceq 1 + \norm{\Psi}_3
\end{equation}
where the constants depend only on $b_0, b_2, b_3$, and $d_0$.
Using the Solobev inequalities (\ref{Ssr}) we estimate
$$
\norm{\Psi}_3^3 \preceq  \norm{\nabla_h \Psi}_2 \>\norm{\Psi}_2^2.
$$
In conjunction with (\ref{abpsi3}), this implies
\begin{equation}\label{nablapsi}
\norm{\nabla_h \Psi}_2 \preceq 1.
\end{equation}
Note also that by means of (\ref{Ssr}), we also have an $L^p$ bound
\begin{equation}\label{Lp.psi}
\norm{\Psi}_p \preceq 1,
\end{equation}
for any $p\ge 2$.

At this point the usual bootstrap approach applies; we sketch the details 
for the sake of completeness.  Assume that from $b_0, b_2, \dots, b_k, d_0$ we have extracted
the bound
\begin{equation}\label{induct}
\bnorm{\nabla_h^{j-2} \Psi}_2 \preceq 1, \qquad\text{for } j=2,\dots,k
\end{equation}
for $k \ge 3$.   (We start the induction at $k=3$ by (\ref{nablapsi}) and (\ref{Lp.psi}).)
Note that at this stage we also have
\begin{equation}\label{induct.phi}
\bnorm{\nabla_g^j \varphi}_2 \preceq 1, \qquad\text{for } j=0,\dots,k.
\end{equation}
(The $\varphi$ estimates stay two derivatives ahead of those for $\Psi$.)
According to \cite[Appendix]{OPS:1988}, the heat coefficient $b_{k+1}$ takes the form
\begin{equation}\label{bk.curv}
b_{k+1} = c_1 \int_X |\nabla_g^{k-1} K(g)|^2\>dg + c_2 \int_X K(g)\>|\nabla_g^{k-2} K(g)|^2\>dg 
+ \text{lower order},
\end{equation}
with $c_1 \ne 0$, where ``lower order'' means fewer derivatives of $K(g)$.  After replacing
$K(g)$ by $\Psi$, the lower order terms can
be estimated directly using the inductive hypothesis (\ref{induct}) and some combination of
(\ref{Lp.psi}) and (\ref{Ssr}).  We can replace $\nabla_g$ by $\nabla_h$ using (\ref{induct.phi})
to estimate the extra terms generated.  Thus from $b_{k+1}$ and the inductive hypothesis
we obtain
\begin{equation}\label{b4.bound}
\bnorm{\nabla_h^{k-1} \Psi}_2^2 \preceq 1 + \int_X |\Psi| \>|\nabla_h^{k-2} \Psi|^2\>dh.
\end{equation}
Applying the H\"older inequality to the second term gives,
$$
\int_X |\Psi| \> \nabla_h^{k-2} \Psi |^2\>dh 
\le \norm{\Psi}_2\> \bnorm{\nabla_h^{k-2} \Psi}_4^2 
$$
Again we turn to (\ref{Ssr}) for the bound
$$
\bnorm{\nabla_h^{k-2} \Psi}_4 \preceq \norm{\nabla_h^{k-1} \Psi}_2^{1/2}\> \norm{\nabla_h^{k-2} \Psi}_2^{1/2}.
$$
By the inductive hypothesis (\ref{induct}) we thus derive from (\ref{b4.bound})
the estimate
$$
\bnorm{\nabla_h^{k-1} \Psi}_2^2 \preceq 1 +\bnorm{\nabla_h^{k-1} \Psi}_2,
$$
which immediately yields
\begin{equation}\label{DPsi.bound}
\bnorm{\nabla_h^{k-1} \Psi}_2 \preceq 1,
\end{equation}
completing the induction.

From the full collection of heat invariants $b_0, b_1, \dots$ we thereby obtain a full
set of $H^k$ estimates:
$$
\bnorm{\nabla_h^{k} \Psi}_2 \preceq 1, \qquad \bnorm{\nabla_h^{k} \varphi}_2 \preceq 1.
$$
To extract $C^k$ estimates is now a simple matter.  Let $P_m$ be an
arbitrary differential operator of order $m$ with coefficients supported in $K \subset X$.
From 
$$
R_h(2)(\Delta_h+2) P_m\Psi = P_m\Psi,
$$ 
we obtain
$$
|P_m\Psi(z)| \le \norm{R(s;z,\cdot)}_2\> \Bigl(\norm{\Delta_h P_m\Psi}_2 + 2 \norm{P_m\Psi}_2 \Bigr)
\preceq 1.
$$

To complete the proof, we must produce a lower bound on the injectivity radius $\inj(X,g)$.
If $K(g) \le 0$, then $\inj(X,g) = \ell_0(g)/2$ and Corollary~\ref{ell0.bound} already 
supplies the estimate.  Otherwise, we have $\kappa := \sup K(g) >0$ and
the $C^0$ bound derived above gives $\kappa \preceq 1$.  In this case the 
result follows from the standard estimate (see e.g.~\cite[\S6.3.2]{Petersen}),
$$
\inj(X,g) \ge \min \left(\frac{\pi}{\sqrt{\kappa}},\> \frac{\ell_0(g)}2\right).
$$
\end{proof}

\end{document}